\documentclass{amsart}
\usepackage{eufrak, amssymb,amsmath, amscd, mathrsfs, yfonts, bbm, graphics, dsfont}

\newtheorem{claim}{}[section]

\newtheorem{thm}[claim]{Theorem}
\newtheorem{lemma}[claim]{Lemma}
\newtheorem{remark}[claim]{Remark}

\newtheorem{cor}[claim]{Corollary}

\newcommand{\lra}{\longrightarrow}

\newcommand{\gH}{\mathfrak{H}}
\newcommand{\cT}{\mathcal{T}}

\newcommand{\cH}{\mathcal{H}}
\newcommand{\cK}{\mathcal{K}}
\newcommand{\cE}{\mathcal{E}}

\newcommand{\Id}{\textbf{Id}}
\newcommand{\cL}{\mathcal{L}}
\newcommand{\cA}{\mathcal{A}}

\newcommand{\gA}{\mathfrak{A}}
\newcommand{\cD}{\mathcal{D}}
\newcommand{\wb}{\widetilde{b}}
\newcommand{\wB}{\widetilde{B}}
\newcommand{\wA}{\widetilde{A}}
\newcommand{\cJ}{\mathcal{J}}

\title [Addition and multiplication of c-free random variables]{A Fock space model for addition and multiplication of c-free random variables}
\author{Mihai Popa}
\address{Center for Advanced Studies in Mathematics at the Ben Gurion University
of Negev, P.O. B. 653, Be'er Sheva 84105, Israel and
\newline
Institute of Mathematics "Simion Stoilow" of the Romanian Academy, P.O. Box 1-764,
Bucharest, RO-70700, Romania}
\email{popa@math.bgu.ac.il}

\begin{document}

\maketitle

\begin{abstract}
 The paper presents a Fock space model suitable for constructions of c-free algebras. Immediate applications are direct proofs for the properties of the c-free $R$- and $S$-transforms. 
 

\end{abstract}

\section{Introduction}
 Two important tools in Free Probability theory are the $R$- and $S$-transforms, that play similar role to Fourier, respectively Mellin transform. More precisely, besides strong regularity properties, if $X$ and $Y$ are two free non-commutative random variables, then $R_{X+Y}(z)=R_X(z)+R_Y(z)$ and $S_{XY}(z)=S_X(z)\cdot S_Y(z)$ if $X,Y$ have non-zero first order moments.
 
 In literature  there  are two main techniques to prove the additive, respectively multiplicative properties of the $R$- and $S$-transforms. The  proofs given by D.-V. Voiculescu (\cite{voiculescu1}, \cite{voiculescu2}) and U. Haagerup (\cite{haagerup}) based on functional analysis techniques, namely on the properties of the annihilation and creation operators on the full Fock space, while the proofs of R. Speicher and A. Nica (\cite{speichernica}) are based on combinatorial techniques on the lattice of non-crossing partitions (also non-crossing linked partitions appear in the proofs for the multiplicative property of the $S$-transform in \cite{dykema-adv}, \cite{mvp2}). 
 
 In early '90's, M. Bozejko, M. Leinert and R. Speicher introduced the notion of \emph{c-freeness}, which extends the notion of freeness to the framework of and algebra endowed with two $(\phi, \psi)$, rather than one, normalized linear functionals (see Section 2 for the exact definitions). (A more general approach to c-freeness, considering pairs of completely positive maps and conditional expectations have been pursued by F. Boca (\cite{boca}), K. Dykema and E. Blanchard (\cite{dykema-blanchard}), M. Popa, V. Vinnikov (\cite{mvp-stoch}, \cite{mvp-vv}) etc). Addition of c-free random variables is studied in \cite{bls}, where is constructed a c-free version of the $R$-transform, the ${}^cR$-transform, with similar additivity and analytic properties (for $\phi=\psi$, the two transforms coincide); multiplication of c-free random variables was studied in \cite{mvp-jcw}, where is constructed a c-free extension of the $S$-transform. In both cases, the proofs of the key properties (addition for the ${}^cR$- and multiplication for the ${}^cS$-transform) are combinatorial, much like the proofs from \cite{speichernica}, heavily relaying on the properties on non-crossing partitions. The present material gives a new approach to c-free random variables, in the spirit of the construction from \cite{haagerup}. Particularly, we give a more direct proof of the additive  and multiplicative properties of the ${}^cR$- and ${}^cS$-transforms, based on the properties of the creation and annihilation operators on a certain type of Fock space.
 
 Besides the Introduction, the paper is organized in 3 sections. Section 2 presents basic definitions, the construction of the space $\cE(\gH, \cT(\cK))$ and an operator algebras  model for c-free algebras. Section 3 presents the construction of some operators of prescribed ${}^cR$ and ${}^cS$-transforms and Section 4 gives the proof for the additive, respective multiplicative properties of ${}^cR$ and ${}^cS$.
 
 In this paper, rather than the $S$- or ${}^cS$-transforms, we use, to simplify the notations their multiplicative inverses, the so-called $T$- and ${}^cT$-transforms -- i. e. $T_X(z)\cdot S_X(z)=1$ (see \cite{dykema-adv}, \cite{mvp2}), respectively  ${}^cT(z)\cdot{}^cS(z)=1$ (see \cite{mvp-jcw}).

  \section{A construction of c-free algebras}

\S1.
  Suppose $\cA$ is a complex unital algebra endowed and $\psi:\cA\lra\mathbb{C}$ is a linear map such that $\varphi(1)=1$. A family of unital subalgebras $\{\cA_i\}_{i\in I}$ of $\cA$ is said to be \emph{free} (with respect to $\psi$) if 
  \[
  \psi(x_1\cdots x_n)=0
 \]
 whenever $x_j\in\cA_{\epsilon(j)}$ with $\epsilon(k)\neq\epsilon(k+1)$ and $\psi(x_j)=0$ for $1\leq j\leq n$ and $1\leq k<n$.
  
  If $\phi:\cA\lra\mathbb{C}$ is another linear map with $\phi(1)=1$, the family $\{\cA_i\}_{i\in I}$ of unital subalgebras of $\cA$ is said to be \emph{c-free} with respect to ($\phi, \psi$) if $\{A_i\}_{i\in I}$ are free with respect to $\psi$ and
  \[
  \phi(x_1\cdots x_n)=\phi(x_1)\cdots \phi(x_n)
  \] 
   whenever $x_j\in\cA_{\epsilon(j)}$ with $\epsilon(k)\neq\epsilon(k+1)$ and $\psi(x_j)=0$ for $1\leq j\leq n$ and $1\leq k<n$.

  Take $X$  an element from $\cA$ and let $m_X(z)=\sum_{k=1}^\infty\psi(X^k)$ denote the moment generating series of $X$ with respct to $\psi$. As formal power series,   the transforms $R_X(z)$ and, if $\psi(X)\neq0$, $T_X(z)$ are defined by the equations
 \begin{eqnarray}
 m_X(z)&=&R_X(z[1+m_X(z)])\label{rdef}\\
 \frac{1}{z}m_X(z)&=&\left[T_X(m_X(z))\right]\cdot(1+m_X(z))\nonumber.
 \end{eqnarray}
We warn the reader that the version of the $R$-trasform that is used in the present material differs from the original definition of D.-V. Voiculescu (that we will call here $\mathcal{R}$) by a multiplication with the variable $z$: $R_X(z)=z\cdot\mathcal{R}_X(z)$. As also seen in \cite{mvp-jcw}, \cite{speichernica}, this shift of coefficients is simplifying the notations in several recurrence relations from \S2.

For $\cH$ a complex Hilbert space, we define $\cT^0(\cH)=\cH\oplus(\cH\otimes\cH)\oplus(\cH\otimes\cH\otimes\cH)\oplus\dots$ and $\cT(\cH)=\mathbb{C}\omega\oplus\cT^0(\cH)$, where $\|\omega\|=1$. For $e\in\cH$ a nonzero vector, the creation operator over $e$, $a^\ast_e\in\cL(\cT(\cH))$, is given by the relations
\begin{align*}
&{a}^\ast_e\omega=e\\
&{a}^\ast_e v_1\otimes v_2\otimes\cdots v_k=e\otimes v_1\otimes v_2\otimes\cdots v_k,\ \text{for}\ v_1,\dots, v_k\in\cH
\end{align*}
while the annihilation operator over $e$, $a_e\in\cL(\cT(\cH))$, is given by
\begin{align*}
&{a}_{e}\omega=0\\
&{a}_{e} v_1\otimes v_2\otimes\cdots v_k=
\langle v_1,\xi\rangle\cdot v_2\otimes\cdots\otimes v_k
\end{align*}

We remind the following result (see \cite{dnv} for (1), \cite{haagerup}, Theorem 2.2 and Theorem 2.3, for (2) and  (3)):
\begin{thm}\label{freeprob}
Let $\cH$ be a complex Hilbert space, $e_1$ and $e_2$ be two orthogonal vectors from $\cH$, and $f_1, f_2$ be polynomials with complex coefficients. For $e\in\cH\setminus\{0\}$ we will denote by $\cA(e)$ the algebra generated by the creation and annihilation operators over $e$.
\begin{enumerate}
\item[(1)]\label{item1} The algebras $\cA(e_1)$ and $\cA(e_2)$ are free with respect to the vacuum state $T\mapsto\langle T\omega, \omega\rangle$.
    \item[(2)] If $\alpha_i=a^\ast_{e_i}+f(a_{e_i})$, $(i=1,2)$, then $R_{\alpha_i}(z)=z\cdot f_i(z)$ and
    $R_{\alpha_1+\alpha_2}(z)=z\cdot f_1(z)+z\cdot f_2(z)$.
    \item[(3)] If $f_i(0)\neq0$, and $\beta_i=[\Id_{\cT(\cH)}+a^\ast_{e_i}]f(a_{e_i})$, $(i=1,2)$, then $T_{\beta_i}(z)=f_i(z)$ and
    $T_{\beta_1\beta_2}(z)=f_1(z)\cdot f_2(z)$.
\end{enumerate}

\end{thm}

\S2.
Consider now two complex Hilbert spaces $\cK$ and $\gH$ and $\omega$  a distinguished unit vector in $\gH$. Take
$\cT(\cK)=\omega_1\oplus\cT^0(\cK)$, where again $\|\omega_1\|=1$ and
\[
\cE(\gH, \cK)=\gH\oplus(\gH\otimes\cT(\cK)).
\]
Later, in Sections 2 and 4, we will consider $\cE(\gH, \cK)$ for a particular $\gH$; when there is no possibility of confusion, to simplify the writting, we will use $\cE$ for $\cE(\gH, \cK)$. Put $\gH^0=\gH\ominus\mathbb{C}\omega$, $\Omega=\omega\otimes\omega_1$, $\cE^0=\cE\ominus\mathbb{C}\Omega$. We define the following embedding $\pi:\cL(\cK)\lra\cL(\cT(\cH))$:
\[
\pi(a)=a\oplus a\otimes(\Id_{\gH\otimes\cT^0(\cK)}\oplus 0_{\mathbb{C}\Omega}).
\]
 Note that $\pi(\cL(\gH))$ has unit $\pi(\Id_{\gH})=\Id_{\cE_0}\neq\Id_{\cE}$.

 For a nonzero vector $\eta\in\cK$ we define the operators $A^\ast_\eta$ and $\{A_{\eta,n}\}_{n\geq0}$ from $\cL(\cE)$ as follows:
 \begin{eqnarray*}
 A^\ast_\eta\zeta&=&0,\ \text{if}\ \zeta\in\gH\oplus\gH^0\otimes\cT(\cK)\\
 A^\ast_\eta\omega\otimes\omega_1&=&\omega\otimes\eta\\
 A^\ast_\eta\omega\otimes\zeta&=&\omega\otimes(\eta\otimes\zeta)\ \text{for all}\ \zeta\in\cT^0(\cK).
 \end{eqnarray*}
 we put $A_{\eta,0}=\Id_{\mathbb{C}\Omega}$ and, for $n\leq1$, we define $A_{\eta,n}$ via
   \begin{align*}
   &A_{\eta,n}\omega\otimes(\eta^{\otimes n})=\omega\otimes\omega_1\\
   &A_{\eta,n}\zeta=0,\ \text{if}\ \zeta\notin\mathbb{C}\omega_1\otimes(\eta^{\otimes n}),\ \text{where $\eta^{\otimes n}=\underbrace{\eta\otimes\cdots\otimes\eta}_{n\ \text{times}}$.}
   \end{align*}

 We will use the notation  $\cD(\eta)$ for the algebra generated by $A^\ast_\eta$ and $\{A_{\eta,n}\}_{n\geq0}$. If $\cA_1$ and $\cA_2$ are two subalgebras of $\cL(\cE)$, then the notation $\cA_1\vee\cA_2$ will stand for the algebra generated by them in $\cL(\cE)$.

\begin{remark}\label{remark1}

 Fix $\eta, \eta_0\in\cK$  unit vectors. From the definitions of $\pi, A^\ast_\eta, A_{\eta,n}$, trivial verifications  give that  that
 \begin{equation*}
 A_{\eta,n}(A^\ast_\eta)^p=
 \left\{
 \begin{array}{ll}
 A_{\eta,n-p}&\text{if}\ n\geq p\\
 0&\text{if}\ n<p\\
\end{array}
 \right.
 \end{equation*}
 If $x\in\pi(\cL(\gH))$ and $n\geq0$, then $xA_{\eta,n}=0$. Also, if $m,n>0$, then
 \begin{align*}
 &A_{\eta,n}A_{\eta_0,m}=0\\
 &\Id_{\mathbb{C}\Omega}A_{\eta,n}=A_{\eta,n};\ \ A_{\eta,n}\Id_{\mathbb{C}\Omega}=0\\
 &A_{\eta,n}\Id_{\cE_0}=0;\ \ \Id_{\cE_0}A^\ast_\eta=A^\ast_\eta\Id_{\cE_0}=A^\ast_\eta.
 \end{align*}

 \end{remark}

\begin{remark}\label{remark2}
For $\eta_1, \eta_2\in\cK$, we have that
\[
\text{Range}\Big((A^\ast_{\eta_1})^pA^\ast_{\eta_2}\Big)=\text{Span}\{\omega\otimes\eta_2\otimes\eta_1^{\otimes p}, \omega\otimes\zeta\otimes\eta_2\otimes\eta_1^{\otimes p}:\ \zeta\in\cT^0(\cK)\},
\]
 therefore, if $\eta_1\perp\eta_2$, then
 \[
 A_{\eta_1,n}(A^\ast_{\eta_1})^pA^\ast_{\eta_2}=0.
 \]
\end{remark}
  On $\cL(\cE)$ we consider the functionals $\phi(\cdot)=\langle \cdot \Omega, \Omega\rangle$ and $\psi(\cdot)=\langle \cdot \Omega_1, \Omega_1\rangle$.

\begin{lemma}\label{lemma1}
 Suppose $\eta_1, \eta_2\in\cK$ and $x\in\pi(\cL(\gH))$. Then
 \begin{eqnarray*}
 A^\ast_{\eta_1}xA^\ast_{\eta_2}&=&A^\ast_{\eta_1}\psi(x)A^\ast_{\eta_2}\\
 A_{\eta_1,n}xA^\ast_{\eta_2}&=&A_{\eta_1,n}\psi(x)A^\ast_{\eta_2}
 \end{eqnarray*}
\end{lemma}

\begin{proof}
 Since $A_{\eta_2}(\cE)=\omega_1\otimes\cT^0(\cK)$,  for any $\zeta\in\cE$ we have that:
 \begin{align*}
  &A_{\eta_2}\zeta=\omega_1\otimes\zeta^\prime\ \text{for some}\ \zeta^\prime\in\cT^0(\cK)\\
  &xA_{\eta_2}\zeta=\psi(x)\omega_1\otimes\zeta^\prime+v\otimes\zeta^\prime\ \text{for some}\ v\in\cT^0(\cH).
 \end{align*}
 But $\cT^0(\cH)\otimes\cT^0(\cK)\subset \ker(A^\ast_{\eta_1}), \ker(A_{\eta_1,n})$ hence the conclusion.
\end{proof}

In the proof of the main result of this section, Theorem \ref{algcfree}, we will use the following lemma:
\begin{lemma}\label{lemma2}
 Suppose $\cA_1$ and $\cA_2$ are two free independent subalgebras of an algebra $\cA$ with respect to some linear map $\varphi$ and $a_0, a_1, \dots, a_{n+1}\in\cA_1$, $b_1, \dots, b_n\in\cA_2$ are such that $\varphi(a_k)=\varphi(b_k)=0$ for all $1\leq k\leq n$. Then
 \[
 \varphi(a_0b_1\cdots b_na_{n+1})=0.
 \]
\end{lemma}
\begin{proof}
 Take $d_j=a_j-\varphi(a_j)$, $j\in\{0, n+1\}$. Then $\varphi(d_j)=0$ and $a_j=\varphi(a_j)+d_j$, therefore
 \begin{eqnarray*}
 \varphi(a_0b_1\cdots b_na_{n+1})&=&\varphi(a_0)\varphi(a_0b_1\cdots b_nd_{n+1})+\varphi(a_0)\varphi(a_0b_1\cdots b_n)\varphi(a_{n+1})\\
 &&+\varphi(d_0b_1\cdots b_nd_{n+1})+\varphi(d_0b_1\cdots b_n)\varphi(a_{n+1})
 \end{eqnarray*}
 and all the above four terms cancel from the definition of free independence.
\end{proof}

\begin{thm}\label{algcfree} Let $\cA_1, \cA_2$ be two subalgebras of $\cL(\gH)$ which are free independent with respect to $\psi$ and let $\eta_1,\eta_2$ be two orthogonal unit vectors $\cK$. Then the  algebras $\gA_1=\pi(\cA_1)\vee\cD(\eta_1)$ and $\gA_2=\pi(\cA_2)\vee\cD(\eta_2)$  are c-free with respect to $(\phi,\psi)$.
\end{thm}
\begin{proof}
  It suffices to prove that for $x_1,\dots, x_m$ such that  $x_j\in\cA(e_{\epsilon(j)}, \eta_{\epsilon(j)})$ with $\epsilon(i)\neq\epsilon(i+1)$ and $\psi(x_k)=0$, we have
 \begin{align}
 &\psi(x_m\cdots x_2x_1)=0 \label{free}\\
 &\phi(x_m\cdots x_2x_1)=\phi(x_m)\cdots\phi(x_2)\phi(x_1)\label{cfree}
 \end{align}
 Note that $\cT(\cH)\otimes\cT(\cK)\perp\Omega_1$ and
\begin{eqnarray*}
\cD(\eta_i)(\cE)&\subseteq&\cT(\cH)\otimes\cT(\cK)\\
\pi(\cA_i)\left(\cT(\cH)\otimes\cT(\cK)\right)&\subseteq&\cT(\cH)\otimes\cT(\cK),
\end{eqnarray*}
hence $\psi$ cancels on all reduced products from $\gA_1\cup\gA_2$ that contain factors from $\cD(\eta_1)$ or $\cD(\eta_2)$. It follows that we only need to prove the relation (\ref{free}) for $x_1,\dots, x_m\in\pi(\cA_1)\cup\pi(\cA_2)$, statement which is equivalent to the free independence of $\cA_1$ and $\cA_2$.

We will prove (\ref{cfree}) by induction on $n$. For $n=1$, the assertion is trivial. For the induction step, it suffices to prove that
\begin{equation}\label{eq11}
\phi(x_n\cdots x_1)=\phi(x_n)\phi(x_{n-1}\cdots x_1).
\end{equation}
 Taking $x^\prime_n=x_n-\phi(x_n)\Id_{\mathbb{C}\Omega}$, we have that $\phi(x^\prime_n)=0$ hence (\ref{eq11}) is equivalent to $\phi(x_n\cdots x_1)=0$ whenever $\phi(x_n)=0$.

 Suppose $x_n\in\cA_1\vee\cD(\eta_1)$. then $x_n$ is a linear combination of monomials in elements from $\cA_1$ and $\cD(\eta_1)$. From Lemma \ref{lemma1}, we can suppose that all factors from $\cD(\eta_1)$ are consecutive, so $x_n$ is a sum of elements from $\cA_1$ and monomials of the types $y^\prime_1(A^\ast_{\eta_1})^pA_{\eta_1,m}y_1$ or $y^\prime_1(A^\ast_{\eta_1})^py_1$, with $y^\prime_1, y_1\in\cA_1\cup{\Id}$ and $p\geq0$. If $y^\prime_1\neq\Id$ or $p\neq0$, then $x_n(\cE)\perp\Omega$, hence $\phi(x_n\cdots x_1)=0$. Also, if $m=0$, then either $y_1=\Id$ and $\phi(x_n)\neq0$ or $y_1\in\cA_1$ and $x_n=0$. Therefore we can suppose that $x_n=A_{\eta_1,m}y_1$ for some $m>0$ and $y_1\in\cA_1$ and all other $x_j$ are either elements of $\cA_{\epsilon(j)}$ or monomials as above.

Let $k=\max\{j: x_j\ \text{contains}\ A^\ast_{\epsilon(j)}\}$ and $p=\max\{j: x_j\ \text{contains}\ A_{\epsilon(j)}\}$. If $p>k$, then $x_n\cdots x_p=A_{\eta_1,m}yA_{\eta_{\epsilon(p)}}y^\prime$,
for some $y\in\cA_1\vee\cA_2$ and
$y^\prime\in\cA_{\epsilon(p)}\vee\cD(\eta_{\epsilon(j)})$. From Lemma \ref{lemma1} and Remark \ref{remark1}, $A_{\eta_1,m}yA_{\eta_{\epsilon(p)}}=
A_{\eta_1,m}\psi(y)A_{\eta_{\epsilon(p)}}=0$.

Suppose that $p\leq k$. If $\epsilon(k)=2$, then $x_n\cdots x_k=A_{\eta_1,m}yA^\ast_{\eta_2}y^\prime$, for some $y\in\cA_1\vee\cA_2$ and $y^\prime\in\cA_2\vee\cD(\eta_2)$.
Applying Lemma \ref{lemma1} and Remark \ref{remark2}, we have
\[
x_n\cdots x_k=A_{\eta_1,m}\psi(y)A^\ast_{\eta_2}y^\prime=0.
\]
 If $\epsilon(j)=1$, then, from Remark \ref{remark2},
 \begin{eqnarray*}
 x_n\cdots x_k
 &=&
 A_{\eta_1,m}y_1x_{n-1}\cdots x_{k+1}y A^\ast_{\eta_1}y^\prime\\
 &=&
 A_{\eta_1,m}\psi(y_1x_{n-1}\cdots x_{k+1}y)A^\ast_{\eta_1}y^\prime
 \end{eqnarray*}
but $\psi(y_1x_{n-1}\cdots x_{k+1}y)=0$ from Lemma \ref{lemma2}, so q.e.d..

\end{proof}

\begin{cor}\label{cor1}
With the notations from \S1, take $\gH=\mathbb{C}\omega\oplus\cT^0(\cH)$, where $\cH$ is a complex Hilbert space of dimension at least 2.

Let $e_1, e_2$, respectively $\eta_1, \eta_2$ be two pairs of orthogonal unit vector from $\cH$, respectively $\cK$. Then the algebras $\cD(\eta_1)\vee\pi(\cA(e_1))$ and $\cD(\eta_2)\vee\pi(\cA(e_2))$ are c-free with respect to the maps $\phi$ and $\psi$ considered above.
\end{cor}
\begin{proof}

 From Theorem \ref{freeprob}(1), the algebras $\cA(e_1)$ and $\cA(e_2)$ are free in $\cL(\cT(\cH))$ with respect to $\langle \cdot\omega, \omega\rangle$, and the conclusion follows from Theorem \ref{algcfree}.

\end{proof}

\section{The ${}^cR$- and ${}^cT$- transforms}

Consider an algebra $\mathcal{A}$ with two states $\phi, \psi:\mathcal{A}\lra\mathbb{C}$ and $X\in\mathcal{A}$. Let $m_X(z)=\sum_{k=1}^\infty\psi(X^k)$, respectively $M_X(z)=\sum_{k=1}^\infty\phi(X^k)$ be the moment-generating series of $X$ with respect to  $\psi$, respectively $\phi$. We define the ${}^cR$-, and, if $\psi(X)\neq0$, the ${}^cT$-transforms of $X$ by the following equations:
\begin{eqnarray}
{}^cR_X(z[1+m_X(z)])\cdot(1+M_X(z))&=&M_X(z)[1+m_X(z)]\label{crdef}\\
\left[{}^cT_X(m_X(z))\right]\cdot(1+M_X(z))&=&\frac{M_X(z)}{z}\label{ctdef}
\end{eqnarray}

 With the notations from Section 2, for $\eta\in\cK$ a non-zero vector and  $f=\sum_{k=0}^N f_k X^k$ a polynomial with complex coefficients, we define $A_{\eta,f^\otimes}$  via:
 \begin{eqnarray*}
 A_{\eta,f^\otimes}&=&\sum_{k=0}^N f_k\cdot A_{\eta, k}\\
 \end{eqnarray*}

\subsection{The ${}^cR$-transform}

\begin{thm}\label{crthm}
Let $\eta$ be a unit vector from $\cK$, $b\in\pi(\cL(\gH))$ and $f=\sum_{p=0}^Mg_p\cdot z^p$  be a polynomial with complex coefficients. Consider $\alpha\in\cL(\cE)$ given by:
\[
\alpha=b+A^\ast_\eta+A_{\eta,f^\otimes}
\]
Then ${}^cR_\alpha(z)=zf(z)$.
\end{thm}
\begin{proof}
It suffices to show that $zf(z)$ satisfies the equation (\ref{crdef}), which is equivalent to the following recurrence
  \begin{equation}\label{recg}
  \phi(\alpha^n)=\sum_{0\leq p \leq n}
  \sum_{\substack{q_1,\dots,q_{p}\geq0\\n\geq 1+p+q_1+\cdots q_{p}}}
  \phi\left(\alpha^{n-1-(p+q_1+\dots+q_{p})}\right)\cdot f_p\cdot \psi(\alpha^{q_1})\cdots \psi(\alpha^{q_{p}})
  \end{equation}
 for all $n>0$.

Let us denote $A=A^\ast_\eta$ and  $B=A_{\eta,f^\otimes}$. The triple $(b, A, B)$ satisfies the following relations:
\begin{align}
&b\Omega=0, B(\cE)=\mathbb{C}\Omega\label{i}\\
&Ab^qA=A\psi(b^q)A, Bb^qA=B\psi(b^q)A\ \text{for all}\ q>0\label{ii}\\
&\phi(BA^n)=f_n,\ \text{for all}\ n\geq0\label{iii}
\end{align}
(equations (\ref{i}) and (\ref{ii}) are consequences of the relations from Remark \ref{remark1}, and (\ref{iii}) follows from Lemma \ref{lemma1}.)

Let $I=\{b, A, B\}$. Since $\alpha=\sum_{x\in I} x$, we have that
 \begin{equation}\label{eq201}
 \phi(\alpha^n)=\sum_{(x_1,\dots, x_n)\in I^n}\phi( x_n x_{n-1}\cdots x_1)
 \end{equation}
To further simplify the writting, we introduce the following notations
 \begin{align*}
 &I[n,j]=\Big\{(x_1,\dots, x_n)\in I^n,\min\{k:\ x_k=B\}=j\Big\}.
 \end{align*}

Since $b\Omega=0$ and $A(\cE)\perp\Omega$, we have that  $\phi(x_n\cdots x_1)=0$ unless $x_n=B$, hence $(x_n, \dots, x_1)\in I[n,j]$ for some $j$.
Also, for $(x_n,\dots,x_1)\in I[n,j]$, since $B(\cE)=\mathbb{C}\Omega$, we have that $x_j\cdots x_1\Omega=\phi(x_j\cdots x_1)$, so
$\phi(x_n\cdots x_1)=\phi(x_n\cdots x_{j+1})\phi(x_j\cdots x_1)$,
 therefore
(\ref{eq201}) becomes
 \begin{eqnarray}
 \phi(\alpha^n)
 &=&
 \sum_{j=1}^n\sum_{(x_1,\dots, x_n)\in I[n,j]}\phi(x_n\cdots x_1)\nonumber\\
 &=&\sum_{j=1}^n\sum_{(x_1,\dots, x_n)\in I[n,j]}\phi(x_n\cdots x_{j+1})\phi(x_j\cdots x_1)\nonumber\\
 &=&\sum_{j=1}^n\sum_{(x_1,\dots, x_j)\in I[j,j]}\phi(\alpha^{n-j})\phi(x_j\cdots x_1).\label{eq202}
 \end{eqnarray}

 Consider $(x_1,\dots,x_n)\in I[n,n]$. If $n=1$, then $\phi(x_n\cdots x_1)=\phi(B)=f_0$. If $n>1$, then $x_1\Omega=0$ unless $x_1=A$. Let $1=k_1<\dots<k_p<n$ be the set of all indices $k$ such that $x_k=A$. Letting $q_{j}=k_{j+1}-k_j-1$,\ $q_p=n-k_{p}-1$ and applying property (\ref{iii}), we obtain:
 \begin{eqnarray}
 \sum_{(x_1,\dots,x_n)\in I[n,n]}\phi(x_n\cdots x_1)&=&\sum_{p=1}^{n-1}\sum_{\substack{0\leq q_1,\dots, q_{p}\\ q_1+\dots q_{p}<n-p}}\phi\left(B\cdot \psi(b^{q_{p}})\cdot A\cdots \psi(b^{q_1})A\right)\nonumber\\
 &=&\sum_{\substack{0\leq q_1,\dots, q_{p}\\ q_1+\dots q_{p}<n-p}} f_p\cdot \psi(b^{q_{p}})\cdots\psi(b^{q_1})\label{eq206}.
 \end{eqnarray}
 Finally, the equality $\psi(b^q)=\psi(\alpha^q)$ and equations (\ref{eq202}), (\ref{eq206}) imply (\ref{recg}), so q.e.d..

\end{proof}

\subsection{The ${}^cT$-transform}

\begin{thm}
Let $\eta$ be a unit vector  from $\cK$, $d\in\pi(\cL(\gH))$ and $f(z)=\sum_{k=0}^Mf_k\cdot z^k$  be a polynomial  with complex coefficients such that  $\psi(b),f_0\neq 0$. Consider $\beta\in\cL(\cE)$ given by:
\[
\beta=d+ d A^\ast_\eta +A_{\eta,f^\otimes}.
\]
Then
 ${}^cT_\beta(z)=f(z)$.
\end{thm}
\begin{proof}
 The proof is similar to the one of Theorem \ref{crthm}. Denote again $A=A^\ast_\eta$, $B=A_{\eta, f^\otimes}$  and consider the sets
  \begin{align*}
  &J=\{d,dA, B\}\\
  &J[n,l]=\{(x_1,\dots,x_n)\in J^n, min\{k:\ x_k=B\}=l\}.
\end{align*}
For $(x_1,\dots,x_n)\in J^n$, we have that $\phi(x_n\dots x_1)=0$ unless $x_n=B$, hence $(x_1,\dots,x_n)\in J[n,l]$ for some $1\leq l\leq n$. Also, note that equation (\ref{eq202}) holds true if we replace $I$ with $J$, therefore
\begin{equation}\label{eq208}
\phi(\beta^n)=\sum_{l=1}^n\phi(\beta^{n-l})\sum_{(x_1,\dots,x_l)\in J[l,l]}\phi(x_l\cdots x_1).
\end{equation}
Fix $n>0$ and let $(x_1,\dots,x_1)\in J[n.n]$. If $n=1$, then $\phi(x_n,\dots,x_1)=\phi(B)=f_0$. If $n=1$, then $\phi(x_n\cdots x_1)$ cancels unless $x_1=dA$.  Let $1=k_1<k_2<\dots<k_p\leq n-1$ be the set of indices $k$ such that $x_k=dA$. Taking $q_{j}=k_{j+1}-k_j-1$  for $1<j<n-1$ and $q_p=n-k_p-1$, and applying property (\ref{iii}), we obtain:

 \begin{eqnarray}
 \sum_{(x_1,\dots,x_n)\in J[n,n]}\phi(x_n\cdots x_1)
 &=&
 \sum_{p=1}^{n-1}\sum_{\substack{0\leq q_1,\dots, q_{p}\\ q_1+\dots q_{p}<n}}
 \phi\left(B\cdot (b^{q_p-1})\cdot dA\cdot b^{q_{p-1}-1}\cdots  b^{q_1-1}\cdot dA \right)\nonumber\\
  &=&
 \sum_{p=1}^{n-1}\sum_{\substack{0\leq q_1,\dots, q_{p}\\ q_1+\dots q_{p}<n}}
 \Big(B\cdot\psi(d^{q_p})\cdot A\cdot \psi(b^{q_{p-1}})\cdots \psi(d^{q_1})A\Big)\nonumber\\
 &=&
 \sum_{p=1}^{n-1}\sum_{\substack{0\leq q_1,\dots, q_{p}\\ q_1+\dots q_{p}<n}}
 f_p\cdot\psi(d^{q_p})\cdots\psi(d^{q_1})\label{eq206}.
 \end{eqnarray}

And the conclusion follows, since (\ref{eq206}), (\ref{eq208}) and the identity $\psi(\beta^q)=\psi(d^q)$  imply the $f(z)$ satisfies (\ref{ctdef}).

\end{proof}

\section{Addition and multiplication of c-free random variables}

\begin{thm}\label{mainthm}Let $\eta_1, \eta_2$ be  orthogonal unit vectors from $\cK$, let $b_1,b_2, d_1,d_2$ be some elements from $\pi(\cL(\gH))$ and let $f_1,f_2, F_1, F_2$ be polynomials with complex coefficients, such that $\psi(d_i\neq0\neq F_i(0)$.  Define ($i=1,2$):
\begin{align*}
&\alpha_i=b_i+A^\ast_{\eta_i}+A_{\eta_i,{f_i}^\otimes}\\
&\beta_i=d_i
+d_i\cdot A^\ast_{\eta_i}
+A_{\eta_i,{F_i}^\otimes}.
\end{align*}

Then ${}^cR_{\alpha_1+\alpha_2}(z)={}^cR_{\alpha_1}(z)+{}^cR_{\alpha_2}(z)$
 and ${}^cT_{\beta_1\cdot\beta_2}(z)={}^cT_{\beta_1}(z)\cdot{}^cT_{\beta_2}(z)$.

\end{thm}

\begin{proof}
 Suppose that $F_1(z)=\sum_{k=0}^M h_k\cdot z^k$ and $F_2(z)=\sum_{k=0}^M l_k\cdot z^k$ (eventually $h_M$ or $l_M$ are zero).

  To prove the first equality, we introduce the notations
 \begin{eqnarray*}
 \wb&=&b_1+b_2\\
 \wA&=&A^\ast_{\eta_1}+A^\ast_{\eta_2}\\
 \wB&=&A_{\eta_1, F_1^\otimes}+A_{\eta_2, F_2^\otimes}
 \end{eqnarray*}

 Trivial verifications show that the triple ($\wb, \wA, \wB$) verify the conditions (\ref{i})--(\ref{iii}), therefore the reccurrence (\ref{recg}) holds true for $\alpha=\wb+\wA+\wB=\alpha_1+\alpha_2$ and $\{g_k\}_{k=1}^M$ the coefficients of $f_1(z)+f_2(z)$, so ${}^cR_{\alpha_1+\alpha_2}(z)=z[f_1(z)+f_2(z)]$, q.e.d..


 For the second equality, we need to prove that $F_1(z)\cdot F_2(z)$ satisfies (\ref{ctdef})  for $\beta=\beta_1\cdot\beta_2$, that is the recurrence formula:
 \begin{eqnarray}
 \phi(\beta^n)&=&\sum_{p=0}^{n-1}\sum_{\substack{q_1,\dots,q_p>0\\q_1+\dots+q_p\leq n}}\phi(\beta^{n-1-(q_1+\dots+q_p)})\cdot
 \Large[g_p\cdot\psi(\beta^{q_p})\cdots\psi(\beta^{q_1})\Large]\nonumber\\
 &=&
 \sum_{m=1}^n\phi(\beta^{n-m})\cdot
 \Big[
 \sum_{p=1}^{m-1}\sum_{\substack{q_1,\dots,q_p>0\\q_1+\dots+q_p<m}}g_p\cdot\psi(\beta^{q_p})\cdots\psi(\beta^{q_1})
 \Big]
 \label{eq401}
 \end{eqnarray}
is verified for $\beta=\beta_1\cdot \beta_2$ and $g_m$ the coefficient of $z^m$ in $F_1(z)\cdot F_2(z)$.

We introduce the notations
 \begin{align*}
 &b=d_1;\  \  \
 d=d_2\\
 &A_1= A^\ast_{\eta_1};\ \ \
 A_2=A^\ast_{\eta_2}\\
 &B_1=A_{\eta_1, F_1^\otimes};\ \ \
 B_2=B_{\eta_2, F_2^\otimes}
 \end{align*}
Then $B_iA_j=0$ whenever $i\neq j$ and $\beta_1=b+bA_1+B_1$, $\beta_2=d+dA_2+B_2$. For $\beta=\beta_1\cdot \beta_2$, we have that
\begin{eqnarray*}
\beta&=&(b+bA_1+B_1)(d+dA_2+B_2)\\
&=&
bd+bdA_2+bB_2+bA_1d+bA_1dA_2+bA_1B_2+B_1d+B_1dA_2+B_1B_2
\end{eqnarray*}
Consider the sets
\begin{align*}
&\cJ=\{bd, bdA_2, bA_1d, bA_1dA_2, bA_1B_2, B_1d, B_1dA_2, B_1B_2\}\\
&\cJ[n,m]=\Big\{ (x_1,\dots,x_n)\in\cJ^n, \min_{k}\{x_k\in\{B_1d,B_1B_2\}\}=m\Big\}\\
&\overline{\cJ}=\cJ\cup\{bA_1dA_2\}.
\end{align*}
Note first that $bB_2=0$; also, since Lemma \ref{lemma1} implies $B_1dA_2=B_1\psi(d)A_2=0$ we have that $\beta=\sum_{x\in\overline{\cJ}}x$, hence
\[
\phi(\beta^n)=\sum_{(x_1,\dots,x_n)\in\overline{\cJ}}\phi(x_n\dots x_1)
\]
If some $x_k$ is $bA_1dA_2$, then $\phi(x_n\cdots x_1)$ cancels, since the vectors from $\cE$ with the $\cT(\cK)$ component containing mixed tensors in $\eta_1$ and $\eta_2$ are cancelled by any $B_i$ ($i=1,2$) and are also orthogonal to $\Omega$. On the other hand,  $\phi(x_n\cdots x_1)$ also cancels if $x_n\cdots x_1\perp\Omega$, that is if $x_n$ does not start with some $B_i$. It follows that only terms having $x_n\in\{B_1d,B_1B_2\}$ contribute to the sum, that is the sum can be taken only for $(x_1,\dots,x_n)\in\cJ[n,m]$, ($1\leq m\leq n$).

Consider now $(x_1,\dots, x_n)\in\cJ[n,m]$. Then $x_m\cdots x_1\Omega=\phi(x_m\cdots x_1)\Omega$, hence $\phi(x_n\cdots x_1)=\phi(x_n\cdots x_m)\phi(x_m\cdots x_1)$, and
\begin{eqnarray}
\sum_{(x_1,\dots,x_n)\in\cJ[n,m]}\phi(x_n\dots x_1)
&=&
\sum_{(x_1,\dots,x_n)\in\cJ[n,m]}\phi(x_n\cdots x_{m+1})\cdot\phi(x_m\cdots x_1)\nonumber\\
&=&\phi(\beta^{n-m})\cdot\Big[\sum_{(x_1,\dots, x_m)\in\cJ[m,m]}\phi(x_m\cdots x_1)\Big]\label{eq402},
\end{eqnarray}
therefore it suffices to prove that the second factors from the right hand sides of (\ref{eq402}) coincides to the second factor of the $m$-th summand in (\ref{eq401}), that is (we use that $g_m=\sum_{p+k=m}l_p\cdot h_k$):
\begin{eqnarray}
\sum_{(x_1,\dots, x_n)\in\cJ[n,n]}\phi(x_n\cdots x_1)
&=&
\sum_{m=0}^{n-1}g_m\cdot\sum_{\substack{q_1,\dots,q_m>0\\q_1+\dots+q_m<n}}\psi(\beta^{q_m})\cdots\psi(\beta^{q_1})\nonumber\\
&=&
\sum_{\substack{p,k\geq0\\p+q<n}}E(p,k)\label{eq404}
\end{eqnarray}
for $E(p,k)=[l_p\psi(\beta^{q_m})\cdots\psi(\beta^{q_1})]\cdot[ h_k\psi(\beta^{s_k})\cdots\psi(\beta^{s_1})$ where the  summation is done over al $p,q\geq0$ such that $p+q<n$ and all $q_1,\dots,q_p, s_1,\dots s_k>0$ with $q_1+\dots q_p+s_1+\dots+s_k=n-1$.

 Consider $(x_1,\dots,x_n)\in\cJ[n,n]$ such that $\phi(x_n\cdots x_1)\neq0$; then $x_1\Omega\neq0$, so $x_1\in\{B_1B_2, bA_1B_2, bdA_2\}$.

Case I: $x_1=B_2B_2$; this imply that $n=1$ (since $x_1$ already starts with a $B_i$) and $\phi(x_1)=\langle B_1B_2\Omega, \Omega\rangle=l_0h_0$, that is (\ref{eq404}) for $n=1$.

Case II: $x_1=bA_1B_2$.

\noindent In this case $x_1\Omega=bA_1B_2\Omega=bA_1\Omega\cdot l_0$.

Let $j=\min\{k:\  x_k\ \text{contains}\ B_1,\ \text{i. e.}\ x_k\in\{B_1B_2, B_1d\}\}$ (since $x_n\in\{B_1B_2, B_1d\}$, the set is not void). Also, $j\neq n$ will contradict the definition of $\cJ[n,n]$, so $j=n$.

Tensors containing $\eta_1$ are canceled by $B_2$ and, as seen earlier, summands with $A_2$ and $A_1$ not separated by $B_1$ do not contribute to the sum, therefore it follows that $x_2, \dots, x_{n-1}$ do not contain $A_2$ nor $B_2$, so they can be only of the types $bd$ and $bA_1d$.

Let $1<k_2<\dots<k_p<n$ be the indices of the factors of type $bA_1d$ and put $k_1=1$ and $k_{p+1}=n$. For $q_i=k_{i+1}-k_{i}$, applying Lemma \ref{lemma1}, we have
\begin{eqnarray}
x_n\cdots x_2bA_1\Omega
&=&
B_1d(bd)^{q_p-1}bA_1d(bd)^{q_{p-1}-1}\cdots(bd)^{q_1-1}bA_1\Omega\nonumber\\
&=&
B_1\psi((db)^{q_p})
A_1\psi((db)^{q_{p-1}})\cdots \psi((db)^{q_1})A_1
\Omega\nonumber\\
&=&
(B_1A^p_1\Omega)\psi((db)^{q_p})\cdots\psi((db)^{q_1})\nonumber\\
&=&h_p\Omega\cdot\psi(\beta^{q_p})\cdots\psi(\beta^{q_1})\label{setting}
\end{eqnarray}
since $B_1A^p_1\Omega=h_p\Omega$ and  $\psi(db)=\psi(bd)=\psi(\beta)$ due to the traciality of the vector states. Multiplying with $l_0$ and summing, we obtain
\begin{equation}\label{eq411}
\sum_{\substack{(x_1,\dots, x_n)\cJ[n,n]\\x_1=dA_1B_2}}\phi(x_n\cdots x_1)
=\sum_{p=1}^{n-1}
E(p,0)
\end{equation}

Case III: $x_1=bdA_2$.

 Let $x_j$ be the factor of the smallest index that contains $B_2$. Since $x_j\in\cJ$, we have that $x_j=y\cdot B_2$, with $y\in\{bA_1, B_1\}$.

 None of the factors $x_2, \dots ,x_{j-1}$ contains $B_1$ (otherwise it will contradict the definition of $J[n,n])$; if some of them will contain $A_1$, then $yB_2x_{j-1}\cdots x_2=0$, since $B_2$ cancels all the tensors mixing $\eta_1$ and $\eta_2$. Hence $x_2, \dots, x_{j-1}\in\{bd, bdA_2\}$. Again, let $1=j_1<\dots<j_k<n$ be the indices of the factors of type $bdA_2$ and put $j_{k+1}=j$. For $s_i=k_{i+1}-k_{i}$, applying Lemma \ref{lemma1}, we have
 \begin{eqnarray}
 x_j\cdots x_1\Omega&=&y\cdot B_2(bd)^{s_k}A_2(bd)^{s_{k-1}}A_2\cdots (bd)^{s_1}A_2\Omega\nonumber\\
 &=&y\Omega\cdot l_k\psi(\beta)^{s_k}\cdots\psi(\beta^{s_1})\label{eq412}
 \end{eqnarray}
(we used that $B_2A_2^k\Omega=l_k\Omega$ and that $\psi(\beta^s)=\psi((bd)^s)$.

If $y=B_1$ then $y\Omega=h_0$; also, the minimality of $n$ implies $j=n$.

If $y=bA_1$, since $x_n\in\{ B_1B_2, B_1d\}$, the minimality of $n$ implies that $x_{j+1}, \dots, x_{n-1}$ do not contain $B_1$. If they will contain $A_2$ or $B_2$, then $\phi(x_n\cdots x_1)=0$ as seen earlier, so they must be of the types $bd$ or $bA_1d$. Also, if $x_n=B_1B_2$, then again $\phi(x_n\cdots x_1)=0$, so we can suppose $x_n=B_1d$.  In this case, $x_n\cdots x_{j=1}y\Omega$ is in the setting of formula (\ref{setting}), so it is computed accordingly to it.

Summing, we obtain
\begin{equation}\label{eq420}
\sum_{\substack{(x_1, \dots, x_n)\in\cJ[n,n]\\ x_1=bA_2}}\phi(x_n\cdots x_1)=\sum_{\substack{p, k-1>0\\p+k<n}}E(p,k)
\end{equation}
and the conclusion follows, since (\ref{eq420}) and (\ref{eq411}) imply (\ref{eq404}).


\end{proof}

\begin{cor} Let $\mathcal{A}$ be a unital algebra, $\Phi, \Psi:\cA\lra\mathbb{C}$ be two linear maps with $\Phi(1)=\Psi(1)=1$ and let $X,Y$ be two c-free (with respect to the maps $\Phi$, $\Psi$) elements from $\cA$.
\begin{enumerate}
\item[(i)]$R_{X+Y}=R_X+R_X$ and ${}^cR_{X=Y}={}^cR_X+{}^cR_Y$ as formal power series.
\item[(ii)]If $\Psi(X), \Psi(Y)$ are nonzero, then $T_{XY}=T_{X}\cdot T_{Y}$ and ${}^cT_{XY}={}^cT_X\cdot{}^cT_{Y}$ as formal power series.
\end{enumerate}
\end{cor}

\begin{proof}

 The equalities for $R$- and $T$-transforms are basic properities in the Free Probability Theory (see \cite{dnv}, \cite{speichernica}). We need to prove (i) and (ii) for the ${}^cR$- and ${}^cT$-transforms.

 As in Corollary \ref{cor1}, we will consider two complex Hilbert spaces $\cH$ and $\cK$ of dimension at least two and $\cE=\cT(\cH)\oplus[\cT(\cH)\otimes\cT(\cK)]$, where
 \begin{align*}
 &\cT(\cH)=\mathbb{C}\omega\oplus\cH\oplus(\cH\otimes\cH)\oplus\dots\\
 &\cT(\cK)=\mathbb{C}\omega\oplus\cK\oplus(\cK\otimes\cK)\oplus\dots
 \end{align*}
We fix  $e_1, e_2$, respectively $\eta_1, \eta_2$ two pairs of orthogonal unit vectors from $\cH$, respectively $\cK$. From Corollary \ref{cor1}, the algebras $\gA_1=\cD(\eta_1)\vee\pi(\cA(e_1))$ and $\gA_2=\cD(\eta_2)\vee\pi(\cA(e_2))$ are c-free with respect to $\phi(\cdot)=\langle\cdot\omega\otimes\omega, \omega\otimes\omega\rangle$ and $\psi(\cdot)=\langle\cdot\omega, \omega\rangle$.

Also note that, from the relations defining the free, respectively c-free independence (see Section 2, \S1), the moments up to order $N$ of $X+Y$ and $XY$ with respect to $\Phi$ and $\Psi$ are uniquelly determined by the moments of order up to $N$ of $X$ and $Y$.

For (i), consider $f_1(z)$, $F_1(z)$, repectively $f_2(z)$, $F_2(z)$ be the polynomials obtained by the trucation of order $N$ of $R_X$, ${}^cR_X$, respectively $R_Y$, ${}^cR_Y$ (i. e. if ${}cR_X(z)=\sum_{k=1}^\infty l_k\cdot z^k$, then $F_1(z)=\sum_{k=1}^N l_k\cdot z^k$ and the analogues).

 With the notations from Section 2, take ($i=1,2$)
 \[
 \alpha_i=\pi(a^\ast_{e_i}+f_i(a_{e_i}))+A^\ast_{\eta_i}+A_{\eta_i,F_i^\otimes}
 \]
 We have that $\alpha_i\in\gA_i$, so $\alpha_1, \alpha_2$ are c-free with respect to $\phi$ and $\psi$, hence, from Theorem \ref{mainthm}, 
 \begin{align}
 &{}^cR_{\alpha_1+\alpha_2}(z)= {}^cR_{\alpha_1}(z)+ {}^cR_{\alpha_2}(z)\label{creq}
 \end{align}
 From Theorem \ref{freeprob}(2) and Theorem \ref{crthm}, we have that
 $R_{\alpha_i}(z)=f_i(z)$
  and  
 ${}^cR_{\alpha_i}(z)= F_i(z)$, 
 therefore $R_{\alpha_1}$, ${}^cR_{\alpha_1}$, and $R_{X}$, ${}^cR_{X}$ coincide up to order $N$. Then equations (\ref{rdef}) and (\ref{crdef}) imply  that the moments up to order $N$ of $X$ with respect to $\Phi$, respectively $\Psi$, coincide to the moments up to order $N$ of $\alpha_1$ with respect to $\phi$ and $\psi$. The same holds true for $Y$ and $\alpha_2$, therefore the moments up to order $N$ of $X+Y$ and $\alpha_1+\alpha_2$ do coincide. Henceforth, from equation (\ref{crdef}), the first $N$ coefficients of ${}^cR_{\alpha_1+\alpha_2}$ and ${}^cR_{X+Y}$ do coincide. Since $N$ is arbitrary, equation (\ref{creq}) gives the conclusion.
 
 The proof for (ii) is similar, taking ($i=1,2$)
 \[
 \beta_i=\pi\Big((\Id+a^\ast_{e_i})f_i(a_{e_i})\Big)+\pi\Big((\Id+a^\ast_{e_i})f_i(a_{e_i})\Big)A^\ast_{\eta_i}+A_{\eta_i,F_i^\otimes}
 \]
 where $f_1, F_1$, respectively $f_2, F_2$ are now the polynomials given by the truncation of order $N$ of $T_X$ and ${}^cT_X$, respectively $T_Y$ and ${}^cT_Y$.
\end{proof}

\end{document}